\documentclass[11pt]{amsart}

\usepackage{amsfonts,epsfig,fancyhdr}
\usepackage{newlfont,amsfonts,amssymb,amsmath,amsthm,amsgen,amscd,datetime,dsfont,hyperref,tensor}
\usepackage[small,nohug,heads=littlevee]{diagrams}
\usepackage[normalem]{ulem}
\diagramstyle[labelstyle=\scriptstyle]
\usepackage{multicol,picinpar,enumerate,cleveref,mathabx}
\usepackage{euscript,color}

\DeclareMathOperator{\arcsinh}{arcsinh}

\newcommand\re[1]{(\ref{#1})}

\renewcommand{\L}{\mathbf{L}}

\def\p{\partial}

\newcommand{\C}{\ensuremath{\mathds{C}}}

\renewcommand{\O}{\ensuremath{\mathrm{O}}}

\def\sh{\mathrm{sinh\,}}
\def\ch{\mathrm{cosh\,}}

\newcommand{\hut}{\widehat}

\newcommand{\bmat}{\begin{pmatrix}}
\newcommand{\emat}{\end{pmatrix}}

\newcommand{\U}{\mathrm{U}}

\newcommand{\e}{\mathrm{e}}

\renewcommand{\d}{{\mathrm d}}
\newcommand{\bcase}{\begin{case}}
\newcommand{\ecase}{\end{case}}

\newcommand{\bclaim}{\begin{claim}}
\newcommand{\eclaim}{\end{claim}}

\newcommand{\bstep}{\begin{step}}
\newcommand{\estep}{\end{step}}

\newcommand{\bhlem}{\begin{hlem}}
\newcommand{\ehlem}{\end{hlem}}

\newcommand{\bleer}{\begin{leer}}
\newcommand{\eleer}{\end{leer}}
\newcommand{\bde}{\begin{definition}}
\newcommand{\ede}{\end{definition}}

\newcommand{\bs}{\begin{proposition}}
\newcommand{\es}{\end{proposition}}
\newcommand{\btheo}{\begin{theorem}}
\newcommand{\etheo}{\end{theorem}}
\newcommand{\bfolg}{\begin{corollary}}
\newcommand{\efolg}{\end{corollary}}
\newcommand{\blem}{\begin{lemma}}
\newcommand{\elem}{\end{lemma}}
\newcommand{\bnote}{\begin{note}}
\newcommand{\enote}{\end{note}}
\newcommand{\bprf}{\begin{proof}}
\newcommand{\eprf}{\end{proof}}
\newcommand{\bd}{\begin{displaymath}}
\newcommand{\ed}{\end{displaymath}}
\newcommand{\be}{\begin{eqnarray*}}
\newcommand{\ee}{\end{eqnarray*}}
\newcommand{\eeqa}{\end{eqnarray}}
\newcommand{\beqa}{\begin{eqnarray}}
\newcommand{\bi}{\begin{itemize}}
\newcommand{\ei}{\end{itemize}}
\newcommand{\bnum}{\begin{enumerate}}
\newcommand{\enum}{\end{enumerate}}
\renewcommand{\la}{\langle}
\renewcommand{\ra}{\rangle}

\newcommand{\beq}{\begin{equation}}
\newcommand{\eeq}{\end{equation}}

\newcommand{\rr}{\mathbb{R}}

\newcommand{\M}{M}

\newcommand{\vf}{\varphi}
\newcommand{\earr}{\end{array}\]}
\newcommand{\barr}{\[\begin{array}}
\newcommand{\bvec}{\left(\begin{array}{c}}
\newcommand{\evec}{\end{array}\right)}

\newcommand{\g}{\mathfrak{g}}

\newcommand{\hol}{\mathfrak{hol}}

\newcommand{\Hol}{\mathrm{Hol}}
\newcommand{\+}{\oplus}

\newcommand{\so}{\mathfrak{so}}
\newcommand{\spin}{\mathfrak{spin}}

\renewcommand{\sp}{\mathfrak{sp}}
\newcommand{\su}{\mathfrak{su}}

\renewcommand{\U}{\mathbf{U}}

\newcommand{\del}{\partial}

\newcommand{\bbem}{\begin{bem}}
\newcommand{\ebem}{\end{bem}}
\newcommand{\bbez}{\begin{bez}}
\newcommand{\ebez}{\end{bez}}
\newcommand{\bbsp}{\begin{bsp}}
\newcommand{\ebsp}{\end{bsp}}

\newcommand{\hm}{\widehat{\M}}

\newcommand{\tg}{\widetilde{g}}

\renewcommand{\gg}{g}
\newcommand{\hg}{\widehat{\gg}}

\newcommand{\bR}{{\mathbb{R}}}

\newcommand{\inter}{\makebox[7pt]{\rule{6pt}{.3pt}\rule{.3pt}{5pt}}\,}

\newcommand{\belabel}[1]{\begin{equation}\label{#1}}

\theoremstyle{definition}
\newtheorem{definition}{Definition}[section]
\newtheorem{bem}[definition]{Remark}
\newtheorem{bez}[definition]{Notation}
\newtheorem{bsp}[definition]{Example}

\newtheorem*{bsp*}{Example}
\newtheorem*{def*}{Definition}
\theoremstyle{plain}
\newtheorem{lemma}[definition]{Lemma}
\newtheorem*{lem*}{Lemma}
\newtheorem{proposition}[definition]{Proposition}
\newtheorem{corollary}[definition]{Corollary}
\newtheorem{theorem}[definition]{Theorem}

\numberwithin{equation}{section}

\begin{document}

\title
{Semi-Riemannian cones}

\thanks{This work was supported by 
the Niels Henrik Abel Memorial Fund in relation to the  2019  
Abel Symposium ``Geometry, Lie Theory and Applications''
 and by
 the Australian Research
Council (Discovery Program DP190102360).
 }



\author{Thomas Leistner}\address{School of Mathematical Sciences, University of Adelaide, SA 5005, Australia}\email{thomas.leistner@adelaide.edu.au}

\subjclass[2010]{Primary 
53C50; Secondary 53C29, 53B30, 53C27}
\keywords{Lorentzian manifolds, pseudo-Riemannian manifolds, metric cones, special holonomy, geodesic completeness, Killing spinors.
}
\begin{abstract}
Due to a result by Gallot a Riemannian cone over a complete Riemannian manifold is either flat or has an irreducible holonomy representation.  This is false in general for indefinite cones but the structures induced on the cone  by holonomy invariant subspaces can be used to study the geometry on the base of the cone. The purpose of this paper is twofold:
first we will give a survey of general results about semi-Riemannian cones with non irreducible holonomy representation and then, as the main result,  we will derive improved versions of these general statements in the case when the cone admits a parallel vector field. We will show that if the base manifold is complete and the fibre of the cone and the parallel vector field have the same causal character, then the cone is flat, and that otherwise, the base manifold admits a certain global warped product structure. We will use these results to give a new proof of the classification results for Riemannian manifolds with imaginary Killing spinors and  Lorentzian manifolds with real Killing spinors which are due to Baum and Bohle.
 \end{abstract}

\maketitle
\setcounter{tocdepth}{1}
\tableofcontents
\section{Introduction}
Given a semi-Riemannian manifold $(M,g)$, the {\em (space-like or time-like) semi-Riemannian cone over $(M,g)$} is the manifold $\hm=\rr_{>0}\times M$ together with the metric
\begin{equation}\label{cone}
\hg_{\epsilon}=\epsilon\, \d r^2+r^2 g,\end{equation}
where $\epsilon =1$ in case of a space-like cone and $\epsilon =-1$ in case of a time-like cone. The original manifold $(M,g)$ is then called the base of the cone.
One reason for considering semi-Riemannian cones is that some systems of PDE on the base correspond to PDE on the cone where they sometimes are easier to study. The key example is the equation for a  {\em  Killing spinor field}, which is an overdetermined system of PDE. A solution to this PDE corresponds to a spinor field on the cone that is parallel for the Levi-Civita connection of the cone metric, which is a closed  system of PDE that can be understood as the prolongation of the original PDE and that is easier to analyse, for example, by  using tools from holonomy theory. Another example is the existence of a Sasaki structure on the base, which correspond to  a K\"{a}hler structure on the cone and hence to a holonomy reduction to the unitary group. Semi-Riemannian cones play also an important role in conformal geometry as conformal ambient metrics for conformal structures containing an Einstein metric. 

As mentioned, the most prominent application is the classification of complete Riemannian manifolds with real Killing spinors by C.~B\"{a}r in \cite{baer93}. He showed that the cone over such manifold  admits a parallel spinor. By a fundamental  theorem of Gallot, the cone is either irreducible or flat. 
With Gallot's result, the holonomy of the cone is one of the irreducible holonomy groups from Berger's list \cite{berger55} that admit invariant spinors \cite{wang89}. This leads to a short list of structures on the cone which correspond to certain structures on the base, all of which had been shown to admit Killing spinors  \cite{BFGK91}. 

In an attempt to apply this method to  Killing spinor (and related) equations on manifolds with {\em indefinite} metrics, in \cite{acgl07,acl19} possible generalisations of Gallot's theorem in the semi-Riemannian context were studied, yielding a comprehensive analysis of the case when the cone admits an invariant subspace under its holonomy representation.   In the first part of this paper, in Section \ref{survsec} we will give a brief survey of these results, including a result from \cite{matveev10}. However instead of providing all the details, we will then focus on the special case when the cone admits a parallel vector field. The focus to this 
this case enables us to show the essential steps in the proofs of the general result without too much technical detail and at the same time give self contained proofs.
More importantly, we will be able  to improve some of the general results  in this special case, in particular in regards to their global character.
In Section \ref{parsec} we will prove the  main result of the paper:

\begin{theorem}\label{introtheo}
Let $(M,g)$ be a geodesically complete semi-Riemannian manifold and let 
$(\hm,\hg_\epsilon)$ be the (time-like or space-like) cone over $(M,g)$. Assume that  $(\hm,\hg_\epsilon)$ admits a parallel vector field $V$.
\bnum
\item If  $\hg(V,V)=\epsilon$, then the cone is flat and $(M,g)$ is of constant curvature $\epsilon$.
\item If $\hg(V,V)=-\epsilon$, then
 $(M,g)$ is globally isometric to 
\[
(\rr\times N, -\epsilon \d s^2+\cosh^2(s)\,g_N),
\]
where $(N,g_N)$ is a complete semi-Riemannian manifold.
\item 
If $\hg(V,V)=0$, then 
$M$ is a disjoint union $M=M_-\cup M_0 \cup M_+$ with $M_\pm$ open and such that $M_0$ is either empty (in which case one of $M_\pm$ is also empty) or a smooth totally geodesic hypersurface and $(M_\pm,g)$ is globally isometric to
\[
(\rr\times N_{\pm} , -\epsilon \d s^2+\e^{ 2 s} g_{N_{\pm}}),
\]
where $(N_{\pm} ,g_{N_{\pm}})$ are complete  semi-Riemannian manifolds. 
Moreover, $M_0=\emptyset$ if and only if $(M,g)$ is Riemannian or negative definite.
 \enum
\end{theorem}

Note that the cases (2) and (3) also include the possibility that $(M,g)$ has constant curvature: in (2)  $g_N$ has constant curvature $\epsilon$ if and only if $g$ also has constant curvature $\epsilon$, whereas in (3),  $g_N$ is flat, if and only if $g$ has constant curvature $\epsilon$ (see \cite[Section 2]{acgl07}).  

The improvements in this theorem over the  of the general result  will allow us to give an alternative proof of the classification of complete Riemannian manifolds with imaginary Killing spinors by Baum \cite{baum89-3} and of Lorentzian manifolds with real Killing spinors \cite{bohle}. In both cases, the parallel spinor on the cone induces a parallel vector field on the Lorentzian cone.  In fact, the results in this paper will be applicable to the classification of Killing spinors whenever the parallel spinor on the cone  induces a parallel vector field. Working out the detail of this is however beyond the scope of this paper.

\subsection*{Acknowledgements}
This paper originated from a talk given at the 
Abel Symposium ``Geometry, Lie Theory and Applications'' in June 2019.
The author would like to thank the organisers for their hospitality and 
the Niels Henrik Abel Memorial Fund for financial support.

\section{Preliminaries}

\subsection{Curvature and geodesics  of semi-Riemannian cones}
Let $\widehat{g}_\epsilon=\epsilon \d r^2 + r^2g$ with $\epsilon=\pm 1$  be the cone metric on
$\widehat{M}=\rr_{>0}\times M$, where $(M,g)$ is a pseudo-Riemannian manifold.
The cone is called space-like if
$\epsilon=1$ and time-like if $\epsilon=-1$.
We denote by $\p_r=\frac{\p}{\p r}$ the radial unit vector field.
The  Levi-Civita connection of the cone $(\widehat{M},\widehat{g}_\epsilon)$ is given by
\beq
\widehat{\nabla}_{\del_r}\del_r  = 0,\qquad
\widehat{\nabla}_X\p_r  = \frac{1}{r} X,\qquad
\widehat{\nabla}_X Y = \nabla_X Y -\epsilon  g(X,Y)\p_r,
\label{lem1}
\eeq
for all vector fields $X,Y\in \Gamma(T\widehat{M})$ orthogonal to $\p_r$.
The curvature $\hut{R}$ of the cone is given by the following formulas
including the curvature $R$ of the base metric $g$:
\beq
\p_r\inter \hut{R}=0,\qquad
\hut{R}(X,Y)Z= R (X,Y)Z - \epsilon \left(  g(Y,Z)X - g(X,Z)Y \right), 
\label{lem2}
\eeq
for $X,Y,Z,\ U\in TM$.
This implies that if $(M,g)$ is a space of constant curvature $\kappa$, i.e.,
\be
 R(X,Y,Z,U) &=& \kappa\left(g(X,U)g(Y,Z) - g(X,Z)g(Y,U)\right),
\ee
then 
the cone has the curvature $r^2\left(\kappa-\epsilon \right) \left(g(X,U)g(Y,Z) - g(X,Z)g(Y,U)\right)$.
In particular, if $\kappa=\epsilon$, then the cone is flat,
as it is the case for the $\epsilon=1$ cone over the standard sphere of radius
$1$ or the $\epsilon=-1$ cone over the hyperbolic space.
%

\bigskip

Let 
$
\widehat{\gamma}=(\rho,\gamma): I\to \hm=\rr_{>0}\times M$ be a geodesic of $(\hm,\hg_\epsilon)$
starting at $\hat p$ and with $\widehat{\gamma}^\prime(0)= a\del_r+ X $. 
The geodesic equations are easily checked to be
\begin{equation}
0= \rho^{\prime\prime}(t) -\epsilon r(t) g\left(\gamma^\prime(t), \gamma^\prime(t)\right) \label{scalar},\qquad
0=2\ \rho^{\prime}(t)\gamma^\prime(t) + \rho(t)\nabla_{\gamma^\prime(t)} \gamma^\prime(t).
\end{equation}
Let $\gamma$ be a
reparametrisation of a geodesic $\beta$ of $(M,g)$,
\begin{equation}\label{geopara}
\gamma(t) = \beta(f(t)),\quad\text{
 with
$\beta(0)=p\, \text{ and }\, \beta^\prime(0)=X$,}\end{equation}
implying the initial conditions 
$f(0)=0$  and $ f^{\prime}(0)=1$ for $f$.

Now let $g(X,X)=c L^2$ with $c\in\{0,\pm 1\}$ and $L>0$.
Hence,  from (\ref{scalar}) we get
\begin{equation}
0= \rho^{\prime\prime}(t) -c  \epsilon \rho(t) f^{\prime}(t)^2L^2, \label{scalar1}\qquad
0=2\ \rho^{\prime}(t)f^{\prime}(t) + \rho(t)f^{\prime\prime}(t)
\end{equation}
with initial conditions
\[
\rho(0)=r,\quad  f(0)=0,\quad
\rho^{\prime}(0)=a,\quad f^{\prime}(0)=1.
\]
If the initial speed $X$ satisfies $c L^2=g(X,X)=0$, i.e., if it is zero or light-like, 
then the equations become
\[
0= \rho^{\prime\prime}(t),\qquad
0=2\ \rho f^{\prime}(t) + (\rho t+ \rho) f^{\prime\prime}(t),
\]
i.e., with solutions  
\begin{equation}\label{geosol0}
\rho(t)=a t+ r,\qquad 
f(t) = \frac{r t}{a t+r}.
\end{equation} 
This implies that $f$ and thus $\hat\gamma$ is defined for $t\in[0,-\frac{r}{a})$ if $a<0$, and for $t\ge 0$ otherwise.

If $cL^2\not=0$, 
the solutions to  equations (\ref{scalar1})  are then given by
\begin{equation}\label{geosol}
\begin{array}{rcl}
\rho (t)&=& \sqrt{(a t+r)^2+c \epsilon L^2r^2t^2},\\[2mm]
f (t)&=&\left\{ \begin{array}{ll}
\frac{1}{L} \mathrm{artan} \left(\frac{L r t}{a t +r}\right),&\text{ if  $c\epsilon =1$,}
\\
\frac{1}{L} \mathrm{artanh} \left(\frac{L r t}{a t +r}\right),&\text{ if $c\epsilon =-1$,}
\end{array}\right.
\end{array}
\end{equation}
This gives us the maximal domain of the cone geodesics under the assumption that $(M,g)$ is complete:
in case of $c\epsilon=1$, in particular if the cone is Riemannian, all  geodesics  are defined on $\rr$ if $a\ge 0$ and on  $t\in [0,-\frac{r}{a})$ if $a<0$.
Otherwise, if 
the functions $\rho$ and $f$ are defined on an interval $[0,T)$, where
$T$ is the first positive zero of the polynomial 
\[
\left(\tfrac{Lrt}{at+r}-1\right) \left(\tfrac{Lrt}{at+r}+1\right)(at+r)^2=
L^2r^2t^2-(at+r)^2
=
((Lr-a)t-r)((Lr +a)t + r),\] 
or $T=\infty$ if the polynomial has no positive zero.
More explicitly, $T=\frac{r}{Lr-a}$ if  $a  < Lr$ and $T=\infty$  if  $a\ge Lr$. We summarise this:

\begin{proposition}\label{geolem}
Let  $(M,g)$ be  a complete semi-Riemannian manifold  and  
 $(\hm,\hg_\epsilon)$  be the  cone.
Let $\hat{p}=(r,p)\in 
\hm$ and $\widehat X=a\del_r|_{\hat p} + X \in T_{\hat p}\hm$ with $g(X,X)=cL^2$ with $c\in\{0,\pm 1\}$ and $L> 0$.
Then there is a geodesic  $\widehat{\gamma}:[0,T)\to \hm$ of $(\hm,\hg)$ starting at $\hat p$  with $\widehat{\gamma}^\prime(0)= \widehat X $ and where 
 \begin{equation}\label{geod}
 T=\left\{
 \begin{array}{rl}
 \infty,&\text{  if $c\epsilon\in \{0,1\}$ and $a\ge 0$, or if $c\epsilon =-1$ and  $a\ge Lr$,}\\[2mm]
 -\frac{r}{a}, &\text{  if $c\epsilon \in \{0,1\}$ and $a< 0$,}\\[2mm]
   \frac{r}{Lr-a},&\text{   if $c\epsilon =-1$ and  $a<Lr$.}\\[2mm]
 \end{array}\right.\end{equation} This geodesic is given by (\ref{geopara}) together with  (\ref{geosol0}) or (\ref{geosol}).
\end{proposition}

\subsection{Completeness of certain warped products}

In this section we are going to study the completeness of warped products of the form 
\[(M=\rr\times N ,g=-\epsilon \d s^2+f^2(s) g_N), \]
where $(N,g_N)$ is a semi-Riemannian manifold,  $f$ is a positive function on $N$ and $\epsilon =\pm 1$. We will need these results in Section \ref{parsec}. 
The Levi-Civita connection of such metrics is given by
\begin{equation}\label{warplc}
\begin{array}{rcl}
\nabla_{\del_s}\del_s&=&0,\\[2mm]
{\nabla}_X \p_s &=& \frac{f^\prime(s)}{f(s)} X,\\[2mm]
{\nabla}_X Y &=& \nabla^N_X Y +\epsilon f^\prime(s)f(s)  g_N(X,Y)\p_s,
\end{array}
\end{equation}

\begin{proposition}\label{completelem}
Let $f:\rr\to \rr_{>0}$ be a smooth function and $(N,g_N)$ be a semi-Riemannian manifold and define
$(M=\rr\times N,g=-\epsilon \d s^2+f^2(s) g_N) $. 
\bnum
\item If all geodesics of $(M,g)$ with initial velocity  tangent to $N$ are defined on $\rr$,  then $(N,g_N)$ is complete. In particular, if $(M,g)$ is complete, 
then $(N,g_N)$ is complete.
\item If $f=\cosh$, then $(M,g)$ is complete if $(N,g_N)$ is  complete.
\item 
If $f(s)=\e^{s}$, then $(M,g)$ is complete if and only if $(N,g_N)$ is  complete and $(M,g)$ is definite, i.e., if $-\epsilon g_N$ is a complete Riemannian metric.
\enum
\end{proposition}
\bprf
(1) Let $(\sigma,\gamma):\rr\to  $ be a geodesic of $(M,g)$ with $\sigma^\prime(0)=0$. Then the geodesic equations are 
\begin{equation}
\label{geoeqs}
\sigma^{\prime\prime} +\epsilon f^\prime(\sigma) f (\sigma) g(\gamma^\prime, \gamma^\prime)=0,\qquad
\nabla^N_{\gamma^\prime}\gamma^\prime+2\frac{f^\prime(\sigma)}{f(\sigma)}\sigma^\prime \gamma^\prime=0,\end{equation}
in particular, $\gamma$ is a pre-geodesic for $g_N$.
The first equation shows that, if $\gamma^\prime(t_0)=0$ for some $t_0$, then $\sigma(t)= at +b$ and $\gamma(t)\equiv \gamma(t_0)$ constant. Hence, if $\gamma^\prime(0)\not=0$, then $\gamma^\prime(t)\not=0$ for all $t$, and so  we can parametrise $\gamma$ by arc-length. The second geodesic equation shows that the reparametrised curve is a geodesic for $g_N$. Hence, $(N,g_N)$ is complete.

(2) 
Assume that  $(M,g)$ is incomplete. Hence there is a maximal geodesic $(\sigma,\gamma):(a,b) \to  M $ with  $b\in \rr$. Then the first geodesic equation in (\ref{geoeqs}) and the equation that the geodesic is of constant length,
\[
-\epsilon (\sigma^{\prime})^2 + f^2(\sigma) g(\gamma^\prime,\gamma^\prime)=c,\]
for a  constant $c$, imply that
\[\cosh(\sigma) \sigma^{\prime\prime} +\sinh(\sigma) (\sigma^\prime)^2 +\epsilon c \sinh(\sigma)=0.
\]
Then, with substituting $\xi=\sinh(\sigma)$, this equation becomes
\[\xi^{\prime\prime} +\epsilon c \xi=0.\]
This is a linear ODE for $\xi$ and hence we can extend $\xi$ and also $\sigma$ beyond $b$ and in fact to $\rr$.   Moreover, 
the second geodesic equation in (\ref{geoeqs})  implies that $\gamma=\beta\circ \tau$, where $\beta $ is a geodesic equation and $\sigma $ and $\tau$ satisfy the equations
\[
\tau^{\prime\prime} +2\,\sigma^\prime \tau^\prime\tanh(\sigma) =0.
\]
With $\sigma:\rr\to \rr$, this is a linear ODE for $\tau$ and hence can be extended beyond $b$. This yields a contradiction to the incompleteness of $(M,g)$.

(3)
Assume that $g$ is complete but indefinite. With $g$  indefinite we can consider a light-like geodesic $(\sigma,\gamma)$, i.e., with 
\[0=-\epsilon(\sigma^\prime)^2 + \e^{2\sigma} g_N(\gamma^\prime,\gamma^\prime).\]
Moreover, from the first geodesic equation
 we obtain 
\[0=\left( (\sigma^\prime)^2 + \sigma^{\prime\prime}\right) = \xi^{-1} \xi^{\prime\prime} ,\]
where we substitute $\xi=\e^{\sigma}>0$. This however yields the equation $\xi^{\prime\prime}=0$, so $\xi$ is affine and, since $(M,g)$ is complete, defined on $\rr$. This contradicts $\xi=\e^\sigma>0$, so $(M,g)$ cannot have light-like geodesics and hence $g$ is definite.

Conversely, assume that $(N,g_N)$ is a complete Riemannian manifold. If $(M,g)$ is not complete, there is a maximal geodesic $\gamma=(\sigma,\beta):[0,b)\to M$ that leaves every compact set in $M$. For such a geodesic we have
\[1=(\sigma^\prime)^2+\e^{2\sigma} g_N(\beta^\prime,\beta^\prime).\]
Hence, with $g_N$ Riemannian, we have 
$0\le (\sigma^\prime)^2\le 1$ and hence that $\sigma$ is bounded on $[0,b)$. This implies that $\sigma$ remains in a compact set, which implies that $\beta$ leaves every compact set in $N$. It also implies that $\e^\sigma$ is bounded away from zero and so  $g_N(\beta^\prime, \beta^\prime)$ is bounded on $[0,b)$ say by $c^2$. 
Then we have that $\beta(t)$ is contained in the geodesic ball around $\beta(0)$ of radius $bc$ since
\[\mathrm{dist}_{g_N} (\beta(0),\beta(t)) \le \mathrm{length}_{g_N} (\beta|_{[0,t]} ) \le bc.
\]
Since $(N,g_N)$ is complete, its geodesic balls are compact, which gives a contradiction. Hence $(M,g)$ is complete.
\eprf

\section{Survey of general results}
\label{survsec}
\subsection {Holonomy groups and  Gallot's Theorem} Let $(M,g)$ be a semi-Riemannian connected manifold. The {\em holonomy group} $\Hol_p(M,g)$ of $(M,g)$ at $p\in M$ is defined as the group of parallel transports, with respect to the Levi-Civita connection of $g$,  along piecewise smooth loops that are closed at $p$. Since the Levi-Civita connection preserves the metric, the holonomy group is a subgroup of the orthogonal groupp $\O(T_pM, g|_p)$ acting on $T_pM$. By fixing a basis of $T_pM$, it can be identified with a subgroup of $\O(r,s)$,  where $(r,s)$ is the signature of $(M,g)$. The holonomy groups at different points in $M$ are conjugated within $\O(r,s)$. Hence the holonomy group as a subgroup in $\O(r,s)$ is well defined up to conjugation and we refer to this as the holonomy group $\Hol(M,g)$.  If $G\subset \O(r,s)$ is a subgroup  and $\Hol(r,s)\subset G$ we say that the holonomy 
reduces to $G$.

The holonomy group is a Lie group. Its connected component is given by parallel transport along {\em contractible} loops. 
Its Lie algebra is denoted by $\hol_p(M,g)$, the {\em holonomy algebra}. 
One can show that the holonomy algebra contains all curvature endomorphisms $R|_p(X,Y)$ at $p$, with $X,Y\in T_pM$  and all derivatives of curvature endomorphisms. Moreover, the Ambrose-Singer holonomy Theorem states that the holonomy algebra at $p$ is spanned as a vector space by the following linear maps, 
\[P_\gamma^{-1}\circ R|_{q}(X,Y ) \circ P_\gamma,
\] where $q\in M$,  $\gamma$ is a path from $p$ to $q$, $P_\gamma$ the parallel transport along $\gamma$ and $X,Y\in T_qM$.

The importance of the holonomy group arises from the well-known holonomy principles. First, parallel sections (with respect to the Levi-Civita connection of $g$) of $TM$ or of any tensor bundle over $M$ are in one-to-one correspondence with vectors (or tensors) that are fixed under the holonomy representation. For example, the existence of parallel vector field reduces the holonomy to a the stabiliser in $\O(r,s)$ of a vector in $\rr^{r,s}$. Another example is the existence of a parallel complex structure, which reduces the holonomy to the unitary group $\U(r/2,s/2)$. 
The other principle is that subspaces in $T_pM$, or in $\rr^{r,s}$, that are invariant under the holonomy group are in one-to-one correspondence with vector distributions that are invariant under parallel transport, or for short, a {\em parallel distribution}. The parallel  distribution $\mathbb{V}\subset TM$ is obtained from an holonomy invariant subspace $E\subset T_pM$ by parallel transport: the fibre  $\mathbb{V}|_q$ is defined as by the parallel transport of $E\subset T_pM$ by any curve from $p$ to $q$.
Because of the holonomy invariance of $E$, this is  a well defined procedure and $\mathbb{V}|_q$ does not depend on the chosen loop. 
A parallel distribution $\mathbb{V}$ is involutive and  defines a foliation of $M$ into totally geodesic leaves of $\mathbb{V}$. 

The holonomy group acts {\em irreducibly} if  it does not admit any invariant subspace. In this case we also say that $(M,g)$ is irreducible.   If $\Hol(M,g)$ does admit an invariant subspace $E$,  since  $\Hol(M,g)\subset \O(r,s)$, the orthogonal space $E^\perp$ is also  invariant under $\Hol(M,g)$. Hence, every holonomy invariant subspace defines two parallel distributions $\mathbb{V}$ and $\mathbb{V}^\perp$.
 If $g$ is indefinite and $E$ is a degenerate subspace, i.e., $E\cap E^\perp\not=\{0\}$, there is a totally light-like distribution $\mathbb{V}\cap \mathbb{V}^\perp$ with totally geodesic leaves. 
If $E$ is non-degenerate, i.e., if $T_pM=E\+E^\perp$, then we also have $TM=\mathbb{V}\+\mathbb{V}^\perp$. In this case we say that the holonomy group acts {\em decomposably}, or for short that $(M,g)$ is decomposable. If there is no non-degenerate subspace that is invariant under $\Hol(M,g)$ we say that the holonomy acts {\em indecomposably}, or that $(M,g)$ is indecomposable.  If $g$ is indefinite, the holonomy group may act indecomposably without acting irreducibly.  This is the case if the holonomy group admits a totally light-like invariant subspace, but no non-degenerate invariant subspace.

If the holonomy group acts decomposably, not just the tangent space decomposes into holonomy invariant subspaces, but under certain global assumptions also the manifold decomposes into a semi-Riemannian product. 
This is due to the  splitting theorems of de Rham \cite{derham52} and Wu \cite{wu64}: if $(M,g)$ is complete and simply connected and the holonomy group acts decomposably,  then $(M,g)$ is isometric to a global semi-Riemannian product $(M_1,g_1)\times (M_2,g_2)$ and the holonomy representation of $(M,g)$ is isomorphic to the product of the holonomy representations of $(M_i,g_i)$. The manifolds $M_i$ correspond to the totally geodesic foliations of $M$ into the leaves of the parallel complementary distributions  $\mathbb{V}$ and $\mathbb{V}^\perp$.  

The notions of irreducibility and (in-)decomposability can also be formulated for the holonomy algebras $\hol(M,g)$, depending on wether the holonomy algebra admits a (non-degenerate) invariant subspace. Note that if $M$ is not simply connected,  the holonomy algebra acting decomposably does not imply that the holonomy group does act decomposably. In particular, the existence of a non-degenerate subspace that is invariant under the holonomy algebra does not necessarily imply the existence of a globally defined parallel distribution. 
  
%
%
%
%
%
%

In regards to the holonomy algebra of a Riemannian cone, Gallot proved the following result:
\begin{theorem}[S. Gallot, \cite{gallot79}]\label{gallottheo}
Let $(M,g)$ be a complete Riemannian manifold of dimension $\ge 2$ such that the holonomy algebra  of  the cone $(\hm,\hg_+)$    does not act irreducibly. 
 Then  $(\hm,\hg_+)$ is flat and hence
 $(M,g)$
has constant curvature $1$. 
If, in addition, $(M,g)$ is simply connected, then $(M,g)$ is isometric to the
standard sphere.
\end{theorem}

We will present Gallot's proof of this theorem in Section \ref{parsec}.
Here we will only explain its first step, which is needed in order to understand possible generalisations and which is based on the aforementioned  {\em holonomy principle}: since the aim is to show that the cone is flat, we can pass to the universal cover, which is the cone over the universal cover of $M$, and assume that the holonomy {\em group} of this cone does admit an invariant subspace $E\subset T_p\hm $. Hence,  this invariant subspace  defines a  vector distribution $\mathbb{V}\subset T \hm$ that is invariant under parallel transport. With $E$ holonomy invariant, its orthogonal space $E^\perp$ is also holonomy invariant and defines a parallel distribution $\mathbb{V}^\perp$.  If the cone is Riemannian, $\mathbb{V}$ and $\mathbb{V}^\perp$  are non-degenerate and hence the tangent space splits into a direct sum of parallel vector distributions $TM=\mathbb{V}\+ \mathbb{V}^\perp$. Both distributions are parallel and hence involutive and define totally geodesic leaves. This splitting and the induced foliation is then used in Gallot's proof.

If the cone metric is indefinite, for example by considering time-like cones over Riemannian manifolds or because already $(M,g)$ is indefinite, a holonomy invariant subspace may be degenerate, i.e., $E\cap E^\perp\not=0$, and hence $\mathbb{V}\cap \mathbb{V}^\perp\not=\{0\}$, so that  the resulting parallel distributions are not complementary.
The following example  shows that  Gallot's Theorem is false in this case.

\bbsp  \label{horosphere} Consider the semi-Riemannian manifold 
\begin{equation}
\label{expwarp}
\left( M=\bR\times N,
g=\d s^2+e^{-2s}g_N\right),
\end{equation} 
where $(N,g_N)$ is a
semi-Riemannian manifold.
Then 
the light-like vector field
\[V=e^{-s}(\p_r+\tfrac{1}{r}\p_s)\] on the time-like cone $(\widehat{M},\hg_{-})$ is parallel.
The manifold $(M,g)$ has constant negative curvature only if $g_N$ is flat. 
If we now assume that $(N,g_N)$ is a complete Riemannian manifold, then, by Proposition \ref{completelem}, $(M,g)$ is a complete Riemannian manifold whose time-like cone $(\hm,\hg_-)$ admits a parallel light-like vector field and hence has a non irreducible holonomy group. However, unless $g_N$ is flat, the cone $\hg_-$ is not flat. This shows that Gallot's Theorem cannot hold when the cone has a parallel light-like vector field.
\ebsp

This example suggests that one has to strengthen the assumptions in Gallot's Theorem in the indefinite setting. To get Gallot's proof started, instead of assuming the existence of some holonomy invariant subspace, on should require  the existence of a {\em non-degenerate} invariant subspace, that gives  complementary parallel distributions $T\hm=\mathbb{V}\+\mathbb{V}^\perp$.
However, the following example shows that such a modification of Gallot's Theorem also fails.

\bbsp\label{ex1} Let $(N,g_N)$ be a complete semi-Riemannian manifold of
dimension at least~$2$ and which is not of constant curvature $1$.
Then the semi-Riemannian manifold 
\begin{equation}
\label{coshmetric}
(M=\bR\times N,g=ds^2+\ch^2(s)g_N)\end{equation}
is complete  by Proposition \ref{completelem}.
Using equation (\ref{warplc}) it is easily established that the spacelike vector field 
\[
V=-\sinh(s)\del_r+\frac{\cosh(s)}{r}\del_s\]
on  the time-like cone $(\hm,\hg_{-})$
is parallel. 
For the  curvature tensor $R$ of $(M,g)$ we have
\[
 R(X,Y)Z=  R_N (X,Y)Z + \tanh^2(s)\left(g_N(Y,Z)X - g_N(X,Z)Y \right),
\]
where $X,Y,Z,\ U\in TF$ and $R_N$ is the curvature tensor  of $(N,g_N)$.
This shows that $(M,g)$ cannot have constant sectional curvature, unless
$N$ has constant curvature $1$.
Thus, in general  the cone
$(\widehat{M},\widehat{g})$ over the complete manifolds $(M,g)$  is decomposable but not flat. 
\ebsp

\subsection{Decomposable cones over complete and over compact manifolds}
In this section we will review a few results that show to which extent Gallot's Theorem generalises to the semi-Riemannian context, having in mind the counter examples of the previous section. We will mainly focus on space-like cones, as the corresponding results for time-like can be obtained by multiplying the cone metric by $-1$.

In this section we will review a few results that show to which extent Gallot's Theorem generalises to the semi-Riemannian context, having in mind the counter examples of the previous section.
First we consider cones over complete semi-Riemannian manifolds.
\btheo[\cite{acgl07}] \label{tmain4} Let $(M,g)$ be a complete  semi-Riemannian  manifold of dimension $\ge 2$ and assume that the holonomy algebra of the cone $(\hm,\hg_+)$ acts  decomposably. 
 Then there exists an open dense submanifold $M'\subset M$ such that each connected component of $M'$   is
isometric to a pseudo-Riemannian manifold of the form
\begin{itemize}
\item[(1)] a pseudo-Riemannian manifold $M_1$ of constant sectional curvature $1$, or
\item[(2)] a pseudo-Riemannian manifold $M_2=\bR_{>0} \times N_1\times N_2$
with the metric
$$-\d s^2+ \ch^2(s) g_1+\sh^2(s)g_2,$$
where $(N_1,g_1)$ and $(N_2,g_2)$ are semi-Riemannian manifolds
and $(N_2,g_2)$ has constant sectional curvature $-1$ or $\dim N_2 \le 1$.

Moreover, the cone $\hm_2$ is isometric to the open subset $\{ r_1>r_2\}$
in   the product of the space-like cone $(\bR_{>0}\times N_1, \d r^2 + r^2 g_1)$
over $(N_1,g_1)$ and the time-like cone
$(\bR_{>0} \times N_2,-\d r^2 + r^2 g_2)$ over $(N_2,g_2)$.
\end{itemize}
 \etheo
 Note that Example \ref{ex1} shows that this theorem is sharp.
 
Next we consider cones over closed semi-Riemannian manifolds $(M,g)$, i.e., when $M$ compact without boundary. Recall that for indefinite metrics compactness of $M$ does not imply the geodesic completeness of $(M,g)$, so we have to assume it, in order to get a version of Gallot's Theorem under these strengthened assumptions.

\btheo[\cite{acgl07}]\label{compacttheo}
Let $(M,g)$ be a closed  and geodesically complete semi-Riemannian
manifold of dimension $\ge 2$. If the cone $(\hm,\hg_+)$ 
is  decomposable, then it is flat 
and hence  $(M,g)$ has constant curvature $1$.
\etheo

Since there is no simply connected  compact indefinite
pseudo-Riemannian manifold
of constant curvature 1, we obtain the following
corollary.

\bfolg[\cite{acgl07}]\label{cor1}
If $(M,g)$ is a simply connected  compact and complete indefinite
pseudo-Riemannian manifold, then the holonomy group of the cone
$(\widehat{M},\widehat{g})$ is indecomposable.
\efolg
Theorem \ref{compacttheo} was strengthened by Matveev in \cite{matveev10}.
\btheo[V.~Matveev \cite{matveev10}]\label{matveevtheo}
Let $M$ be a closed manifold.
\bnum
\item If $g$ is a light-like complete indefinite semi-Riemannian metric on $M$, then the cone $(\hm,\hg_+)$ is  indecomposable.
\item 
If $g$ is a Riemannian metric on $M$, then the cone $(\hm,\hg_-)$  is indecomposable.
\enum
\etheo
Note that (2) in the Theorem \ref{matveevtheo} implies that even though the time-like cone over a compact quotient $M=\mathbb{H}^n/\Gamma$ of hyperbolic space $\mathbb{H}^n$ is flat, its holonomy group acts indecomposably.

\subsection{Local structure of non irreducible cones}
In  this section we will review some results about the local structure of non irreducible cones. We start with decomposable cones.

\btheo[\cite{acgl07}]
Let $(M,g)$ be a  semi-Riemannian manifold such that the holonomy algebra of the  cone $(\hm,\hg_+)$ acts decomposably.
Then 
there exists an open dense submanifold $M'\subset M$ such that
any point $p\in M'$ has a neighborhood $U$ that is isometric to a
semi-Riemannian manifold of the form
$(a,b)\times N_1 \times N_2$ with the metric given either by
\begin{equation}\label{local}
g_+=\d s^2+ \cos^2(s) g_1+\sin^2(s)g_2\quad\text{or}\qquad g_-=-
\d s^2+ \ch^2(s) g_1+\sh^2(s)g_2,\end{equation}
where $g_1$ and $g_2$ are metrics on $N_1$ and $N_2$ respectively.

Moreover, $\bR_{>0}\times U\subset\widehat{M}$ with the metric $g_\pm$ in  (\ref{local}) 
is locally isometric to the product of cone metrics
 $$(\d r_1^2+r_1^2g_1)+(\pm \d r_2^2+r_2^2g_2).$$
\etheo
Note that this theorem also applies to the Riemannian context. The cone over the incomplete Riemannian metric $g_+$ in (\ref{local}),  with $g_1$ and $g_2$ Riemannian,  is decomposable without being flat.

Next we consider the case when the holonomy of the cone admits an invariant degenerate subspace $E$. This implies the existence of an invariant subspace $E\cap E^\perp$ that is totally light-like. We restrict ourselves to the case when the  dimension of $E\cap E^\perp$  is $1$ or $2$. In this case we have a parallel distribution of totally light-like lines or planes. 
%
%
%
%
%

 \btheo[\cite{acl19}]\label{localtheointro}
 Let $(\hm,\hg_-)$ be the time-like cone  over a semi-Riemannian manifold $(M,g)$. 
 If  the cone admits a parallel light-like line field $\L$, then locally there is a parallel trivializing section of $\mathbf{L}$. Moreover, on a dense open subset $\hm_{\mathrm{reg}} \subset \hm$, the metric 
 $\hg$ is   locally  isometric to a warped product of the form 
\begin{equation}\label{ncone}\tg_0=2\, \d u\, \d v+u^2\gg_0,\end{equation}
with a semi-Riemannian metric $\gg_0$, and the metric $\gg$ is locally of the form
\[\gg=\d s^2+\mathrm{e}^{2s}\gg_0.\]
\etheo

The results in the case when $E\cap E^\perp$ is of dimenion $2$ are more technical and related to the existence of s shearfree, geodesic, light-like congruence on the base:
\btheo[\cite{acl19}]\label{planetheorintro}
The time-like cone $(\hm,\hg)$ over a semi-Riemannian manifold $(M,g)$ admits a parallel, totally  light-like $2$-plane field if and only if, locally over an open dense subset, 
the base $(\M,\gg)$  admits two vector fields $V$ and $Z$ satisfying
\begin{equation}\label{gVZint}
\gg(V,V)=0,\ \ \gg(Z,Z)=1,\ \ \gg(V,Z)=0,\end{equation}
and such that
 \begin{eqnarray}
  \nabla_XV
  =
  \alpha (X)V + \gg(X,V)Z,
  &&
    \label{nabZint}
    \nabla_XZ
  =
  -X +\beta(X)V+\gg(X,Z)Z,
  \end{eqnarray}
 with $1$-forms $\alpha $ and $\beta$ on $\M$.
 In particular,  the base $(\M,\gg)$ admits a geodesic, shearfree light-like congruence defined by $V$.
 \etheo
Note that the first equation in equation \re{nabZint} implies that $V^\perp$ is integrable. This allows us to  determine the local form of the metrics with vector fields $V$ and $Z$ satisfying equations (\ref{gVZint}) and (\ref{nabZint}):

\bs[\cite{acl19}]
A semi-Riemannian metric $(M,g)$ admits vector fields $V$ and $Z$ with (\ref{gVZint}) and (\ref{nabZint}) if and only if $(M,g)$ is locally of the form 
$M = M_0 \times \rr^3$ and 
\[ g = \d s^2 + e^{-2s}g_0(u) + 2\, \d u\,\eta,\label{metric:equ}\]
for a family of metrics $g_0(u)$ on $M_0$ depending on $u$ and 
a $1$-form $\eta$ on $M$ such that $\eta (\partial_t)$ is nowhere vanishing 
satisfying the following system of first order PDEs: 
 \begin{equation}\begin{array}{rcl}
  \partial_t\eta_t=\partial_s\eta_t\ =\ X\eta_t\ =\   \partial_t(\eta(X)) &=&0,
  \\
  \partial_t\eta_s&  =&  2\eta_t,
  \\
  \partial_s\,\eta (X) -X\ \eta_s&=&-2\eta (X)
   \label{system:equ-int}\end{array}\end{equation}
 for all $X\in \Gamma(TM_0)$ and where we denote $\eta_t=\eta(\partial_t)$ and $\eta_s=\eta(\partial_s)$.
\es
One can solve  explicitly  the system \re{system:equ-int} in the following way:
Let $f_1=f_1(u)$ be an arbitrary nowhere vanishing smooth function on 
the real line equipped with the coordinate $u$ and $f_2= f_2(x,s,u)$ an arbitrary smooth function on $M$ which does not 
depend on $t$. Let $h_i=h_i(x,s,u)$ be a ($t$-independent) solution of the ordinary differential equation 
\[ \partial_sh_i +2h_i= \partial_if_2\]
for all $i=1,\ldots ,n_0$, where $\partial_i = \partial /\partial x^i$. Then 
\[ \eta_t := f_1(u),\quad \eta_s :=2tf_1(u)+f_2(x,s,u),\quad \eta (\partial_i) := h_i(x,s,u)\]
solves  \re{system:equ-int} and every solution is of this form. 

This provides us 
with a construction method of metrics whose cone admits a totally light-like $2$-plane.

 \bbem
For completeness we should mention further results in \cite{acgl07} for the case when the cone admits a holonomy invariant maximal isotropic subspace $\mathbb{V}=\mathbb{V}^\perp$ and an invariant maximally isotropic complement. This is equivalent to the existence of a para-K\"{a}hler structure on the cone. In \cite[Section 8]{acgl07} we have shown that the existence of a para-K\"{a}hler structure on the cone over $(M,g)$ is equivalent to the existence of a  para-Sasaki structure on $(M,g)$ and a similar correspondence for para-hyper-K\"{a}hler structures on the cone and   para-$3$-Sasakian structures on $(M,g)$.
\ebem
 \subsection{Holonomy of cones}
In the last part of this survey section we are going to review results about the possible holonomy groups of cones.  We will consider the fundamental cases when the 
holonomy group acts irreducibly or  not irreducibly but indecomposably. 

 \subsubsection{Irreducible cone holonomies}
In the case when the holonomy algebra of the cone is irreducible, we can use Berger's list and single those out that can be cone holonomies. They key here is to observe that $\del_r\inter \hut{R}=0$ prevents cones from being  Einstein with non zero Einstein constant. This shows, for example, that cones cannot be irreducible locally symmetric spaces.  So by ruling out all holonomy groups of Einstein cones, we obtain:

\btheo[\cite{acl19}]
\label{irredtheointro}
If $(\hm,\hg)$ is a time-like cone with irreducible holonomy algebra $\mathfrak{g}$, then $\g$ is isomorphic to one of the following Lie algebras
\begin{equation}\label{berger0}
\begin{array}{rclrclrcl}
&&\so(t,s),
&
\mathfrak{u}(p,q),\ \su(p,q) &\subset &  \so(2p,2q),
&
 \sp(p,q) &\subset & \so(4p,4q),
 \\
\so(n,\C)&\subset & \so(n,n),
& 
\g_2^\C&\subset& \so(7,7),
& \spin(7,\C)&\subset&\so(8,8),
 \\
 &&&
\g_2&\subset &\so(7), 
& 
\spin(7) &\subset & \so(8),
\\
&&&
\g_{2(2)} &\subset & \so(3,4),
&
\spin(3,4) &\subset &  \so(4,4).
\end{array}\end{equation}
\etheo

\subsubsection{Holonomy of non irreducible, indecomposable cones}
In general the classification of non irreducible, indecomposable  holonomy groups is widely open and only solved in Lorentzian and in some special cases in signature $(2,n)$ and $(n,n)$. We will focus here in the case where the invariant totally light-like subspace has dimension $1$. The key here is the result in Theorem \ref{localtheointro}, where it was shown that that a cone that admits a parallel light-like line distribution is locally isometric to a metric of the form (\ref{ncone}).

\btheo[\cite{acl19}]\label{holtheointro}
Let $(N,\gg_0)$ be a semi-Riemannian  manifold in dimension $n$ and $\tg_0$ the metric  defined in (\ref{ncone}).
  If the holonomy of $\tg$ acts indecomposably, then 
\begin{equation}\label{holindec}
\hol(\tg_0)\ \subset\  \hol(\gg_0)\ltimes\rr^{t,s}
\end{equation}
where $ \hol(\gg_0)\ltimes\rr^{t,s}$ is a subalgebra of the stabiliser algebra of $\del_v$ in $\so(t+1,s+1)$, i.e., in $\so(t,s)\ltimes\rr^{t,s}\ =\ \so(t+1,s+1)_{\del_v}$, and 
 \[
\mathrm{pr}_{\so(t,s)}(\hol(\tg_0))\ =\ \hol(\gg_0).\]
There is an equality in (\ref{holindec}) whenever $(N,g_0)$ is one of the following:
\bnum
\item
 an irreducible locally symmetric space, or a product thereof;
\item a Riemannian manifold;
\item  a Lorentzian manifold without a parallel light-like vector field. 
\enum
\etheo

%
%
%

\section{Semi-Riemannian cones with parallel vector fields}
\label{parsec}

In this section we will consider the special case when the invariant subspace under the holonomy group of the cone is given by a parallel vector field, that is, the rank of the invariant distribution is one and the distribution admits a global parallel section. For this special case we will prove  versions of the theorems in the previous section that are slightly stronger and more specific, and we will prove Theorem  \ref{introtheo}. Before we do this we will review Gallot's original proof of his theorem  in order to see when we can generalise it to the case of a parallel vector field. We will see that this can be done when the radial vector field and the parallel vector field have the same causal character.

\subsection{The proof of Gallot's Theorem}
 Gallot's proof of  Theorem \ref{gallottheo}
 uses the following fundamental observation, which holds not only for Riemannian cones.
\begin{lemma} 
\label{gallotlem1}Let $(\hm,\hg_\epsilon)$ be semi-Riemannian cone and let $\mathbb{V}\subset T\hm$ be a non degenerate, parallel distribution. Let $p\in M$ such that 
$\del_r|_p\in \mathbb{V}|_p$ and $N^\perp|_p$ the leaf of $\mathbb{V}^\perp$ through $p$. Then the 
image in  $N^\perp_p$ under the exponential map restricted to $\mathbb{V}^\perp_p\in T_p\hm $ is flat.
\end{lemma}
\bprf
Let $\widehat{\gamma} =(\rho,\gamma) :I\to \hm=\rr_{>0}\times M$ be a geodesic in  $(\hm,\hg)$ with $\widehat{\gamma}(0)= p$ and $\widehat{\gamma}^\prime(0)\in \mathbb{V}^\perp|_p$.
 It is easy to check using (\ref{lem1}) that the vector field 
\begin{equation}
\label{ptvf}
F(t)=\rho(t)\del_r- t\widehat{\gamma}^\prime(t)\end{equation}
is parallel transported along $\hat \gamma$.  
Then with $F(0)=r(p)\del_r|_{\hat p}\in \mathbb{V}|_{\hat p}$ we have that
$F(t)\in \mathbb{V}|_{\widehat{\gamma}(t)}$ for all $t$. Since the curvature tensor leave parallel distributions invariant and because of $\del_r\inter \widehat{R}=0$, we have that
\[\widehat{R}(X,Y) F(t)=\widehat{R}(X,Y)\widehat{\gamma}^\prime(t) \in \rr\cdot \mathbb{V}|_{\widehat{\gamma}(t)} \]
for all $t$. On the other hand we have that $\widehat{\gamma}^\prime(t)\in \mathbb{V}^\perp_{\hat \gamma(t)}$ for all $t$. Hence, with $\mathbb{V}\cap \mathbb{V}^\perp=\{0\}$ this implies   
that 
\[\widehat{R}(X,Y)\widehat{\gamma}^\prime(t)=0,\]
for all vector fields $X$ and $Y$ along $\widehat{\gamma}$ and all $t\in I$. 
From this we see that the Jacobi fields along $\widehat{\gamma}$ are those of a flat
manifold, which implies that $N$ is flat.
\end{proof}

Using this lemma, we can now proceed with the proof of Gallot's Theorem.
\begin{proof}[Proof of Theorem \ref{gallottheo}]
By passing to the universal cover of the cone, which is the cone over the universal cover of $M$, we can assume that $\hm$ is simply connected. 
Let $\mathbb{V}$ be a parallel distribution in $T\hm$ and $\mathbb{V}^\perp$ the orthogonal distribution that are induced by the subspace that is invariant under the holonomy group. If we assume that $\mathbb{V}$ is non degenerate, as we can in the case of a Riemannian manifold $(M,g)$, we have $T\hm=\mathbb{V}\+ \mathbb{V}^\perp$. For a given point $p\in \hm$ denote by $N_p$ and $N_p^\perp$ the totally geodesic leaves of $\mathbb{V}$ and $\mathbb{V}^\perp$. Moreover, denote 
\[C=\{p\in \hm\mid \del_r|_p\in \mathbb{V}|_p\},\quad C_\perp=\{q\in \hm\mid \del_r|_q\in \mathbb{V}^\perp|_q\}.\]
Note that $\widehat{\nabla}_X\del_r=\frac{1}{r}X$ for $X\in TM$ implies that neither $C$ nor $ C_\perp$ can contain an open set and hence that $\hm_0=\hm\setminus (C\cup C_\perp)$ is dense in $\hm$. 

\begin{lemma}
Let $(\hm,\hg)$ be a Riemannian cone over a complete Riemannian manifold $(M,g)$. Then for each point  $x\in \hm\setminus (C\cup C_\perp)$ there is a $p\in C$ and a $q\in C_\perp$ such that $x$ lies in the image of the exponential map  $\exp_p$ restricted to $\mathbb{V}^\perp|_p$ and in the image of $\exp_q$ restricted to $\mathbb{V}|_q$.
\end{lemma}
\bprf
Let $x\in M$ and assume that $x\not\in C\cup C_\perp$. Let $\del_r|_x=V+W$ with $V=\mathrm{pr}_{\mathbb{V}|_x}(\del_r|_x)\in \mathbb{V}|_x$ and $W=\mathrm{pr}_{\mathbb{V}^\perp|_x}(\del_r|_x)\in \mathbb{V}^\perp|_x\not=0$. Then
\[\hg(\del_r,W)=\hg(V+W,W)=\hg(W,W)
\]
and
\[\hg(V,V)=  \hg(\del_r-W,\del_r-W)=1-\hg(W,W), \]
which implies that  
\begin{equation}\label{crucial}
0<\hg(\del_r,W)<1.\end{equation}
Let   $\hat\gamma=(\rho,\gamma)$ be the maximal geodesic starting at $x$ with $\rho(x)=r$, satisfying  the initial condition
\[\widehat{\gamma}^\prime(0)=-rW=- r \mathrm{pr}_{\mathbb{V}^\perp|_x}(\del_r|_x).\]
Now we have $a=\rho^\prime (0)= -r \hg(\del_r,W) $, and hence, by the previous section,  the maximal geodesic is defined for $t<T$ with 
\[T=-\frac{r}{a}=\frac{1}{\hg(\del_r,W) }>1,\]
by (\ref{crucial}).
Let $F(t)$ be the parallel transported vector field defined in (\ref{ptvf}) along $\hat \gamma$.
Then 
\[F(0)+ \widehat{\gamma}^\prime(0) =r\del_r|_x - 
r \mathrm{pr}_{\mathbb{V}^\perp|_x}(\del_r|_x)
 \in \mathbb{V}|_x.\] 
 The parallel transport of this vector up to $t=1$ is 
 \[ F(1)+\hat\gamma(1)=r(\hat{\gamma}(1)) \del_r|_{\hat{\gamma}(1)},\]
 which is in $\mathbb{V}|_{\hat\gamma(1)}$ as $\mathbb{V}$ is a parallel distribution. This implies that $\hat\gamma(1)\in C$. 
 
 The argument for $C_\perp$ works completely analogously. 
\eprf
Both lemmas imply that each point in $\hm \setminus(C\cup C_\perp)$ lies in the intersection of  two flat leaves of $\mathbb{V}$ and $\mathbb{V}^\perp$ and hence has a flat neighbourhood. This implies that $\hg$ on $\hm\setminus (C\cup  C_\perp) $ is flat. Since $\hm \setminus (C\cup  C_\perp) $ is dense in $\hm$, this implies that $(\hm,\hg)$ is flat. This finishes the proof of Theorem \ref{gallottheo}.
\end{proof}

\subsection{A generalisation of Gallot's Theorem}

Let $(\hm,\hg_\epsilon)$ be a time-like or space-like cone over a  semi-Riemannian $(M,g)$. 
From now on we restrict to the case when $\mathbb{V}=\rr\cdot V$, where $V$ is a parallel vector field, normalised such that   \[\hg(V,V)=\nu\in \{ -1,0,1\}.\] 
Since $V$ is assumed to be parallel, the leaves of $\rr \cdot V$ are flat, so in order to generalise Gallot's Theorem  we would need to show that the leaves of $ V^\perp$ are also flat. In order  show this using Gallot's method, we need that the set 
\[C=\{p\in \hm\mid \del_r|_p\in \mathbb{R}V|_p\}\] is not empty.
This however can {\em only} be the case  when $\del_r$ and $V$ have the same causal character, i.e., only when $\epsilon=\nu$, i.e.,
\[
C\not=\emptyset\text{ implies }\epsilon=\nu.\]  
We have already seen Examples \ref{horosphere} and \ref{ex1}, which show that 
Gallot's Theorem does not generalise when this condition is not satisfied, i.e., when 
$\epsilon\not=\nu$. We will deal with this case in the next section. 
Here we consider the case when $\nu=\epsilon$.
In this special case we obtain  a generalisation of Gallots Theorem as a  stronger version of Theorem \ref{tmain4}.

\begin{theorem}\label{Veps}
Let $(M,g)$ be a complete semi-Riemannian manifold and let $(\hm,\hg_\epsilon) $ be the cone over $(M,g)$. If $(\hm,\hg)$ admits a parallel vector field $V$ with $\hg(V,V)=\epsilon$, then the cone is flat and $(M,g)$ is of constant curvature $\epsilon$.
\end{theorem}

\bprf
Let $V$ be the parallel vector field on $(\hm,\hg_\epsilon)$ with $\hg (V,V)=\epsilon$. 
As in the proof of Theorem \ref{gallottheo} we consider the set $C=\{ p\in \hm\mid \del_r|_p =  \rr\cdot V|_p\}$ and show that each $q\in \hm\setminus C$ admits a flat neighbourhood. Let $\del_r|_q= \alpha V+W$ with $W\in V^\perp$ and, since $\hg(V,V)=\epsilon$, with $\alpha=\epsilon \hg(V,\del_r)$.
Again we have
\[\hg(\del_r,W)=\hg(W,W)=w\not=0 ,\]
and 
\[\alpha^2\epsilon = \hg(\del_r-W,\del_r-W)=\epsilon-\hg(W,W)=\epsilon - w.\]
Hence we obtain 
\begin{equation}
\label{crucial1}
0<\alpha^2=1 -\epsilon w.\end{equation}
On the other hand we write 
\[W=\epsilon \hg(\del_r,W)\del_r +W_0= \epsilon w \del_r+W_0,\]
with a $W_0\in T_qM$. Hence,
\[\hg(W,W)= \hg(W,W)^2\epsilon+\hg (W_0,W_0),\]
and hence 
\begin{equation}
\hg (W_0,W_0)=w(1-\epsilon w).
\end{equation}
Now let $\hat\gamma$ be a geodesic starting at $q$ with $r(q)=r$ and with $\hat\gamma^\prime(0)=-rW$. We will show that $\hat\gamma$ is defined on $[0,1]$.  We have $\hat\gamma^\prime(0)=a\del_r- rW_0$ with (as in Lemma \ref{geolem})
\[a=-\epsilon w r, \qquad c L^2= r^2g(W_0,W_0)=\hg(W_0,W_0)=    w(1-\epsilon w),\]
with $c=\pm 1$.  We now consider the  cases $c\epsilon=1$, $c\epsilon=-1$ and $c=0$.

If $c\epsilon =1$, then 
\[0<L^2 =\epsilon w (1-\epsilon w),
\]which, together with \ref{crucial1} implies that $\epsilon w>0$ and $a=-\epsilon w r<0$. By Lemma \ref{geolem}, $\hat \gamma$ is defined for $t<T$ with 
\[T=-\frac{r}{a}=\frac{1}{\epsilon w} >1,\]
because of (\ref{crucial1}).

If $c\epsilon=-1$ we get 
\[0<L^2=-\epsilon w(1-\epsilon w),\]
and hence that
$a=-\epsilon w r>0$ and moreover 
\[ r^2L^2 =a(a+r)=a^2+ar  >a^2.\]
Hence, we are in the case $a<rL$ in Lemma \ref{geolem}, and $\hat\gamma$ is defined for $t<T$ with 
with  $T= \frac{r}{Lr-a}$. We show now that $T>1$. For this note that by the previous displayed equation we have
\[L^2r^2-(r+a)^2=a^2 +ra -(r+a)^2 =-r(r+a)<0
\]
since $a>0$. This shows that $Lr<r+a$ and therefore $T= \frac{r}{Lr-a}>1$.

Finally, in the case $c=0$ we must have $w=0$ and hence $a=0$, so $\hat \gamma$ is defined on $[0,\infty)$.

Now we proceed in the proof of Theorem \ref{gallottheo}: the vector field $F(t)$ along $\hat\gamma$ satisfies $F(0)+\hat\gamma(0) = r\alpha V|_q$  
whose parallel transport is given by $F(1)-\hat\gamma(1)=r(\hat\gamma(1))\del_r|_{\hat\gamma(1)}$. This implies that $\hat\gamma(1)\in C$ and by Lemma \ref{gallotlem1} the leaf of $V^\perp$ though $q$ is flat. Since $V$ is a parallel vector field, this implies that $q$ has a flat neighbourhood and hence, since $\hm\setminus C$ is dense, that $(\hm,\hg)$ is flat. 
\eprf

\subsection{Non flat cones with parallel vector field}
Recall the two Examples \ref{horosphere} and \ref{ex1}. 
We will now show that cone with parallel vector fields satisfying the condition $\hg(V,V)\not=\epsilon$ are always of the  form as in these examples and thus obtain a 
stronger version of Theorem~\ref{tmain4} in the case of a parallel vector field on the cone.

First we define the function $u=
\hg (V,\del_r)$ and observe:

\begin{lemma}\label{Vlem}
Let $V$ be a parallel  vector field on the cone $(\hm,\hg_\epsilon)$ over a (not necessarily complete) semi-Riemannian  manifold $(M,g)$. 
Then $u =\hg(V,\del_r)$ is a smooth function on $M$, $u \in C^\infty(M)$,   that satisfies
\begin{equation}
\label{vus}
V=\epsilon u\del_r+\frac{1}{r}\nabla u,\end{equation}
where $\nabla u$ is the gradient of $u$ with respect to $g$, that satisfies
\begin{equation}
\label{nabu}
\nabla \d u=-\epsilon u g.
\end{equation}
\end{lemma}

\bprf With $u =\hg(V,\del_r)$, 
we  split $V$ as $V=-u\del_r+W$ where $W$ is a section of $TM\to \hm$. Since $V$ is parallel, we use (ref{lem1}) to get $0=\hat{\nabla}_{\del_r}V$ which implies that 
$\del_r(u)=0$ and $[\del_r,W]+\frac{1}{r}W=0$. The latter implies that $W=\frac{1}{r}U$ with $U\in \Gamma(TM)$ is a vector field on $M$.
The equation $\nabla V|_{TM}=0$ implies that $\nabla u= U$, where $\nabla u$ denotes the gradient of $u$ with respect to $g$,  and $\nabla \nabla u= -\epsilon u \mathrm{Id}$, i.e., that $\nabla  \d u = \epsilon u g$.
\eprf

Recall that in the  case when  $\nu=\hg (V,V)=0$ or $\nu=-\epsilon$ we have that \[C=\{p\in \hm \mid \del_r|p\in \rr\cdot V|_p\}=\emptyset\] and also that the set of critical points of 
 $u $ is empty, 
\begin{equation}\label{crit}
C_0=\{p\in M\mid \nabla u|_p=0\} =
\emptyset.
\end{equation}
Moreover we have
\begin{equation}
\label{gvv}
g(\nabla u,\nabla u)=
\left\{
\begin{array}{ll}
-\epsilon u^2,&\text{ if }\nu=0,
\\
-\epsilon (1+u^2),&\text{ if }\nu=-\epsilon.
\end{array}\right.
\end{equation}
Then we can show:
\btheo\label{-pvf-theo}
Let $(M,g)$ be a complete semi-Riemannian manifold and $(\hm,\hg_\epsilon )$ be the  cone over $(M,g)$. If  $(\hm,\hg)$ admits a 
parallel  vector field $V$ with $\hg(V,V)=-\epsilon$, then $(M,g)$ is globally isometric to 
\[
(\rr\times N, -\epsilon \d s^2+\cosh^2(s) g_N),
\]
where $(N,g_N)$ is a complete semi-Riemannian manifold.
\etheo

\bprf
The idea is to rescale the gradient $\nabla u$ in a way that the rescaled vector field is a geodesic gradient vector field. To this end consider the function 
$s=-\epsilon \arcsinh\circ u$ on $M$, i.e., $u(p)=\sinh (-\epsilon s(p))$, for which we write $u=\sinh (-\epsilon s)  $. Then we have 
\[\nabla u|_p= -\epsilon \cosh( s(p))\nabla s|_p ,\]
and hence 
\[g(\nabla u, \nabla u)= \cosh^2(s) g (\nabla s,\nabla s)  = (1+\sinh^2(s)) g (\nabla s,\nabla s)= (1+u^2) g( \nabla s,\nabla s).\]
Hence, from (\ref{gvv}) we get $g(\nabla s,\nabla s)=-\epsilon$, so $S=\nabla s $ is a unit gradient vector field. Moreover, 
 from (\ref{nabu}) we get 
\[-\epsilon \sinh(-\epsilon s)X=
\nabla_X\nabla u=
 \sinh(s) g(X,S) S  -\epsilon \cosh (s)\nabla_XS,
\]
and hence
\begin{equation}\label{nabs}
\nabla_XS = \tanh(-\epsilon s)\left(  X + \epsilon g(X,S)S\right).
\end{equation}
This implies that $S$ 
is a geodesic vector field.
Since $(M,g)$ is assumed to be complete, the flow  $\phi $ of $S$ is defined on $\rr\times M$. By the above observation (\ref{crit}) we have  $\nabla u\not=0$    and  hence  all level sets are smooth hypersurfaces. 
Moreover the Lie derivative of $\d s$ in direction of $S$ vanishes,
 \[\cal L_S\d s(X)= \d^2s(X)+X(\d s(S))=X(g(S,S))= 0.\]
 This implies that the flow of $S$ maps each level set of $s$ to a level set of $s$.
 
 For a fixed $p\in M$ we define the function $\sigma(t)=s(\phi_t(p)$. Since $S=\nabla s$ is complete, $\sigma$ is defined on $\rr$ and satisfies the differential equation
 \[\sigma^\prime(t)= \d s|_{\phi_t(p)}(S)=g_{\phi_t(p)}(S,S)\equiv -\epsilon.\]
 Hence $\sigma (t) = -\epsilon t+s(p)$, which shows that
 \begin{equation}\label{level}
 \phi_t(N_c)=N_{-\epsilon t +c},
 \end{equation}
 where $N_c=s^{-1}(c)$ denotes the level set of $s$,  
 Now set
 $N=N_0=\{u=0\}=\{s=0\}$, which is a smooth hypersurface and denote by $g_N$ the restriction of $g$ to $N$. We define a smooth map
 \[
 \Phi:\rr\times N\ni (t,p)\mapsto \phi_t(p)\in M.
 \]
 which, because of (\ref{level}), has the inverse
 \[
 \Phi^{-1}(q)=\left(s(q),\phi_{-s(q)}(q)\right)\in \rr\times N.\]
 This shows that $\phi$ is a diffeomorphism. 
 
 Finally, equation (\ref{nabs}) implies that 
 \[\cal L_Sg(X,Y)=2\tanh(-\epsilon s) g(X,Y),\]
 for all $X,Y\in S^\perp$, i.e., al $X, Y$ tangent to the level sets of $s$. 
 This shows that 
 \[\Phi^*g =-\epsilon \d s^2+\left( \cosh(s)\right)^{2} g_N.\]
 Since $(M,g)$ was assumed to be complete $(N,g_N)$ has to be complete by Proposition \ref{completelem}.\eprf

Now, let $(M,g)$ be a  semi-Riemannian manifold and $(\hm,\hg_\epsilon )$  be the  cone over $(M,g)$. We 
 consider the case that $V$ is a parallel light-like vector field. Recall that in this case we have, in addition to Lemma \ref{Vlem}, that $g(\nabla u,\nabla u)=-\epsilon u^2$. 
In this situation we observe:
\begin{lemma}\label{0geolem}
If $\gamma:I\to M$ is a  geodesic  on $(M,g)$ with $g(\gamma^\prime(0),\gamma^\prime (0))=-\epsilon$  and $f=u\circ \gamma$, then $f^{\prime\prime}=f$, i.e.,
\begin{equation}\label{fgeo}
f(t)=
u(\gamma(0)) \cosh(t)+  g(\nabla u|_{\gamma(0)},\gamma^\prime(0))\sinh(t).
%
\end{equation}
In particular, if $(M,g) $ is complete, then the image of $u$ contains $(0,\infty) $ if $\{u>0\}\not=\emptyset$ and  $(-\infty, 0)$ if $\{u<0\}\not=\emptyset$.
\end{lemma}
\bprf
With $f=u\circ \gamma$ we have 
$f^\prime = g|_\gamma (\nabla u|_\gamma,\gamma^\prime)$ and hence  by Lemma \ref{Vlem},
\[f^{\prime\prime}
=
 g|_\gamma(\nabla_{\gamma^\prime}\nabla u,\gamma^\prime )=-\epsilon f  g|_\gamma(\gamma^\prime, \gamma^\prime)= f.
 \]
 The general solution to this equation is given by (\ref{fgeo}). If $(M,g)$ is complete, the maximal geodesics through a point with $u(p)\not=0$ are defined on $\rr$ and hence, by choosing a geodesic with $\gamma^\prime(0)=\frac{1}{u(p)} \nabla u|_p$, i.e., with $g(\nabla u,\gamma^\prime(0))= -\epsilon u(p)$, we get
 \begin{equation}\label{geo+}
 f(t) = \frac{u(p)}{2}\left( (1-\epsilon)\e^t+(1+\epsilon)\e^{-t}\right)=u(p)\e^{-\epsilon t}.\end{equation}
This implies the statement about  the  image of $u$.
\eprf

\btheo\label{0pvf-theo}
Let $(M,g)$ be a complete semi-Riemannian manifold and $(\hm,\hg_\epsilon )$ be the  cone over $(M,g)$. If  $(\hm,\hg)$ admits a 
parallel  light-like vector field $V$, then $M$ is a disjoint union $M=M_-\cup M_0 \cup M_+$ 
with $M_\pm$ open and such that $M_0$ is either empty (in which case one of $M_\pm$ is also empty) or a smooth totally geodesic hypersurface and $(M_\pm,g)$ is globally isometric to
\[
(\rr\times N_{\pm} , -\epsilon \d s^2+\e^{ 2 s} g_{N_{\pm}}),
\]
where $(N_{\pm} ,g_{N_{\pm}})$ are complete semi-Riemannian manifolds. 
Moreover, $M_0=\emptyset$ if and only if $(M,g)$ is Riemannian.
\etheo

\bprf Recall that for in the case of $V$ being light-like we have that $g(\nabla u,\nabla u)= -\epsilon u^2$. The proof is  analogous to the previous proof, with a difficulty arising from the possibility that the set 
\[M_0=\{p\in M\mid g(\nabla u|_p,\nabla u|_p)=0\} = \{p\in M\mid u(p)=0\} \]
may be non empty, so that the geodesic gradient vector field  $S$ from the previous proof may not be defined on all of $M$. However,  since $V$ is light-like, we have  $\nabla u\not=0$, and so $M_0$   is either empty or  a smooth hypersurface. In fact, if $M_0\not=\emptyset$, it is totally geodesic: if $X\in TM_0=\nabla u^\perp|_{M_0} $, then Lemma \ref{0geolem} shows that $f(t)\equiv 0$, so the geodesics  starting in direction of $M_0$ remain in $M_0$.

We set $M_\pm=\{\pm u>0\}$.  Without loss of generality, we assume that $M_+\not=\emptyset$, in which case we get that $N_+=\{u=1\}\not=\emptyset$ by the previous lemma.

We  consider the function 
$s=- \epsilon \ln \circ (\pm u)$ on $M_\pm $, i.e., $u=\pm \e^{-\epsilon s}$. Then we have 
\[\nabla u= \mp \epsilon \e^{-\epsilon s} \,\nabla s ,\]
and hence, for $S=\nabla s$, 
\[g(\nabla u, \nabla u)= \e^{- 2\epsilon s} g (S,S)  = u^2 g(S,S),\]
and so $g(\nabla s,\nabla s)=-\epsilon$ by (\ref{gvv}). Next we get from (\ref{nabu}) that 
\[\mp \epsilon \e^{- \epsilon s}X=
\nabla_X\nabla u=
\mp \epsilon  \e^{-\epsilon s}  \left( - \epsilon 
g(X,S) S  +  \nabla_XS\right),
\]
and hence
\begin{equation}\label{nabse}
\nabla_XS =  \left(  X + \epsilon g(X,S)S\right).
\end{equation}
Again, this shows that $S$ 
is a geodesic vector field on $M_\pm $. 
Equation \ref{geo+} in the proof of Lemma \ref{0geolem} then shows that the geodesics with initial speed given by $S|_p$ for $p\in  M_\pm $ remain in $M_\pm$ for all $t\in  \rr$. Hence $S$ is a complete vector field on $M_\pm$ with its flow defined on $\rr\times M_\pm$, so we can continue with the proof as for the previous theorem yielding 
  a diffeomorphism 
 \[
 \Phi_\pm:\rr\times N_{\pm} \ni (t,p)\mapsto \phi_t(p)\in M,
 \]
 where $N_{\pm} =\{p\in M\mid u(p)=\pm 1\}=\{p\in M_\pm\mid s(p)=0\}$, with the inverse 
 \[
 \Phi^{-1}_\pm(q)=\left(s(q),\phi_{-s(q)}(q)\right)\in \rr\times N_{\pm}.\]

Now equation (\ref{nabse}) implies that 
 \[\cal L_Sg(X,Y)= 2 g(X,Y),\]
 for all $X,Y\in S^\perp$, i.e., al $X, Y$ tangent to the level sets of $s$. 
 This shows that 
 \[\Phi^*_\pm g =-\epsilon \d s^2+\e^{ 2 s} g_{N_\pm}\]
with a semi-Riemannian manifold  $(N_{\pm},g_{N_\pm})$. In order to conclude that $(N_{\pm},g_{N_\pm})$ are complete, we observe that (\ref{fgeo}) in Lemma \ref{0geolem} shows that geodesics of $(M,g)$ with initial speed tangent to $N_{\pm}$, i.e., with initial speed orthogonal to $\nabla u|_{N_\pm}$, remain in $M_{\pm}$ and hence, because $(M,g)$ is complete, are defined on $\rr$.  With this, Proposition \ref{completelem} implies that $(N_{\pm},g_{N_\pm})$ are complete.

For the last statement, first  note that if $(M,g)$ is Riemannian, then, since $\nabla u\not=0$, we get that  $M_0=\emptyset$. On the other hand assume that $M_0=\emptyset$ and without loss of generality that $M_+=M$, so that globally $(M=\rr\times N,g= -\epsilon \d s^2+\e^{ 2 s} g_{N}))$. 
By Proposition \ref{completelem}, the metric $g$ is only complete if it is definite and $g_N$ is complete.\eprf
%
%
As a corollary we obtain a global version of Theorem \ref{localtheointro}.
\begin{corollary}\label{wavecor}
Let $(M,g)$ be a complete semi-Riemannian manifold and $(\hm,\hg_\epsilon )$ be the  cone over $(M,g)$. If  $(\hm,\hg)$ admits a 
parallel  light-like vector field $V$, then $\hm$ is a disjoint union $\hm=\hm_-\cup \hm_0 \cup \hm_+$ with 
\[\hm_{\pm}=\{p\in \hm \mid \pm \hg (V,\del_r) >0\},\quad \hm_0=\{p\in \hm \mid \hg(V,\del_r)=0\}=\rr_{>0}\times M_0\] and such that  $(\hm_\pm,\hg)$ is globally isometric to
\[
(\rr_{+}\times \rr_{\epsilon}\times N_{\pm} , \widetilde{g}=2\d u \d v +u^2g_{N_{\pm}}),
\]
where $(N_{\pm} ,g_{N_{\pm}})$ are a complete semi-Riemannian manifolds and where $\rr_\pm=\{x\in \rr\mid \pm x>0\}$. 
The isometry is given by 
\[\Psi_{\pm} : \hm\ni (r,s,p)\mapsto
(u=r\e^{ s},v=\frac{\epsilon}{2} r\e^{ -s}, p)\in (\rr_{+}\times \rr_{\epsilon}\times N_{\pm}).
\]
\end{corollary}
\bprf 
We have 
$u^2=r^2\e^{ 2s}$ and  
\[2 \d u \d v= \epsilon \left( \e^{ s} \d r \pm r\e^{ s}ds \right)\left(  \e^{- s} \d r\mp r\e^{- s} \d s\right)=\epsilon (\d r^2 - r^2 \d s^2). \]
Hence, by the previous theorem,  $\Psi_\pm^*\widetilde{g}= \hg$.
\eprf

\section{Lorentzian cones and applications to Killing spinors}
\label{spinsec}
\subsection{Parallel spinors and Killing spinors}
\label{spinsubsec}
Let $(M,g)$ be a semi-Riemannian spin manifold, i.e., a space and time oriented semi-Riemannian manifold with a spin structure,  and let $\Sigma $ its complex spinor bundle. This is a complex vector bundle that is equipped with the following structures:
\bnum
\item 
the Clifford multiplication 
\[ TM\otimes \Sigma \ni X\otimes \varphi\mapsto X\cdot \varphi \in \Sigma,\]
\item a  hermitian bundle metric $\la.,.\ra\in \Gamma(\Sigma^*\otimes \overline{\Sigma}^*)$ on $\Sigma$, conjugate-linear in the second component, that is positive definite if $g$ is Riemannian and of neutral signature if $g$ is indefinite, 
\item the lift $\nabla^\Sigma$ of the Levi-Civita connection to $\Sigma$,
\enum
that satisfy the following properties, where $r$ is the number of negative eigenvalues of $g$, 
\begin{equation} \begin{array}{rcl}
(X \cdot Y + Y \cdot X) \cdot \varphi &=& - 2 \,g(X,Y) \, \varphi,
\label{spin-calc-1}
\\
\langle X \cdot \varphi, \psi \rangle_\Sigma &=&  (-1)^{r+1} \langle
\varphi,
X \cdot \psi \rangle_\Sigma , \\
\nabla^\Sigma_Y(X \cdot \varphi) &=&  (\nabla_YX) \cdot \varphi + X \cdot
\nabla^\Sigma_Y \varphi , \\
X ( \langle \varphi, \psi \rangle_\Sigma ) &=& \langle \nabla^\Sigma_X \varphi
, \psi \rangle_\Sigma + \langle \varphi, \nabla^\Sigma_X \psi \rangle_\Sigma.
\end{array}
\end{equation}
The second of these relations together with $\la.,.\ra$ being Hermitian shows that 
to each spinor field $\varphi$ one can assign a (real) vector field 
$V_{\varphi}\in \Gamma(TM)$
\[ g(V_{\varphi},X) := i^{r+1}\langle \varphi, X \cdot \varphi\rangle_\Sigma \quad \mbox{ for all $X\in TM$}.\]
This vector field is sometimes called the 
 {\em Dirac current of $\varphi$}.
 The above relations also show that $\nabla V_\vf=0$ if $\vf$ is a parallel  a parallel spinor field, i.e., if $\nabla^\Sigma \vf=0$. However, $V_\vf$ be identically zero even if $\vf$ is not. This happens for example for parallel spinors on Riemannian manifolds.
 
 Moreover, the Ricci tensor of a semi-Riemannian manifold with parallel spinor satisfies $g(Ric(X),Ric(X))=0$. In particular, Riemannian manifolds with parallel spinors are Ricci-flat.

A {\em Killing spinor with Killing number $z\in \mathbb{C}$} is a spinor field $\vf\in \Gamma(\Sigma)$ that satisfies the equation
\[\nabla^\Sigma_X\vf=z\, X\cdot \vf.\] Using the above formula one can show that the scalar curvature of a semi-Riemannian manifold with a  Killing spinor is equal to $4n(n-1)z^2$. This implies that $z$ is either real or imaginary and hence the scalar curvature is a  positive or negative constant. 
A Killing spinor with Killing number $z=\pm \frac{1}{2}$ is called  {\em real Killing spinor} and with  $z=\pm \frac{\mathrm{i}}{2}$, $\vf$ an {\em imaginary Killing spinor}. 
Moreover, Riemannian manifolds with Killing spinor are Einstein, so Riemannian manifolds with real/imaginary Killing spinor provide  examples  of Einstein manifolds with positive/negative  scalar curvature. The question which Einstein manifolds (or constant scalar curvature manifolds) can be constructed in this way lead to the problem of classifying manifolds with Killing spinors. 
The fundamental observation for solving this problem is the relation to semi-Riemannian cones:
\btheo[\cite{baer93,bohle}]\label{spincone}
Let $(M,g)$ be a semi-Riemannian spin manifold that  admits a Killing spinor with Killing number  $\pm \frac{\sqrt{\epsilon}}{2}$ if and only if the semi-Riemannian cone $(\hm,\hg_\epsilon)$ admits a parallel spinor field.
\etheo

\bbem[\cite{bohle}]\label{spinwarp}
In \cite{bohle} Bohle proved a more general result:
Let $(M,g)$ be a semi-Riemannian spin manifold and $f:I\to \rr$ be a smooth function. Then the warped product metric 
\[g_{\epsilon, f}=\epsilon \d s^2+ f^2(s) g\] 
on $I\times M$  admits a Killing spinor with Killing number $\hat{\lambda}\in \{0, \pm \frac{1}{2},\pm \frac{\mathrm{i}}{2}\}$ if and only if
\bnum
\item 
The warping function satisfies the ODE $f^{\prime\prime}= -4\epsilon \hat\lambda^2 f$, and 
\item 
$(M,g)$ admits a Killing spinor with Killing number $\pm\lambda$, where $\lambda^2=\hat\lambda^2f^2+\frac{\epsilon}{4} (f^\prime)^2$.
\enum
\ebem
Theorem \ref{spincone}  together with Gallot's Theorem  \ref{gallottheo} was used by B\"ar \cite{baer93} to derive a classification of complete Riemannian manifolds with real Killing spinors: if $(M,g)$ admits a real Killing spinor, the cone admits a parallel spinor and under the assumption of completeness, by Gallot's theorem, the cone is irreducible. Then by Berger's classification of irreducible holonomy groups \cite{berger55}, Wangs classification of those admitting an invariant spinor \cite{wang89} under their spin representation, and the correspondence between holonomy groups and geometric structures, B\"ar arrived at the following classification:
\btheo[C.~B\"{a}r \cite{baer93}]\label{baertheo}
 Let $M$ be a complete, simply connected Riemannian spin manifold with a real Killing spinor. Then $M$ is isometric to round sphere, or  ta a compact Einstein space with one of the following structures: 
Sasaki,  $3$-Sasaki,
  $6$-dimensional  nearly-K\"ahler, 
   or nearly parallel $G_2$.
\etheo
 Baum gave a classification of Riemannian manifolds with imaginary Killing spinors \cite{baum89-3}. Baum's proof does not use the cone construction of Theorem \ref{spincone} explicitly. 
In other signatures the classification of semi-Riemannian manifolds with Killing spinors is  only known in special cases: for example,   Bohle and Baum classified Lorentzian manifolds with real Killing spinors \cite[Section 5]{bohle}, with an addition made in  \cite[Proposition 7.1]{baum00a}, again without using the cone construction explicitly. In the next section we will use our results from the previous section to obtain Baum's and Bohle's  classification results.

\subsection{Lorentzian cones and Killing spinors}
In this section we will use our results of Section \ref{parsec} to derive the classification of complete Riemannian manifolds with imaginary Killing spinors and of complete Lorentzian manifolds with real Killing spinors. In both cases Theorem \ref{spincone} yields a parallel spinor on a Lorentzian cone and hence a parallel Dirac current by the observations in Section \ref{spinsubsec}. In Lorentzian signature one can show that the Dirac current is a causal vector field:
\blem\label{diraclemma}
Let $\vf$ be a parallel spinor field on a spin Lorentzian manifold $(M,g)$. Then $V_\vf$ is a causal  parallel vector field, i.e, $V_\vf\not=0$, $\nabla V_\vf=0$ and $g(V_\vf,V_\vf)\le 0$.
\elem
 
\bprf
We have already seen that $V_\vf$ is parallel, so it is either identically zero or non vanishing and we have to verify its causal character. Since $(M,g)$ is time orientable  we fix a time-like unit vector field $T$ and split $V_\vf\not=0$ as 
\[V_\vf= -g(T,V_\vf)T+ g(N,V_\vf)  N,\]
where $N$ is a spacelike unit normal field orthogonal to $T$.
Then we have by (\ref{spin-calc-1}) that
\[g(V_\vf,V_\vf)= -g(T,V_\vf)^2+ g(N,V_\vf) ^2= - \la T\cdot \vf,\vf\ra^2+  \la N\cdot \vf,\vf\ra^2,\]
and we have to show that this is not positive. For this observe that the endomorphism $T\cdot N$ on $\Sigma$ squares to the identity by the defining relation for the Clifford algebra in  (\ref{spin-calc-1}),
\[ T\cdot N \cdot T\cdot N= - N\cdot T\cdot T\cdot N=  - N\cdot N =1.\] 
Hence $T\cdot N$ has eigenvalues $\pm 1$ and we can split $\vf=\vf_++\vf_-$ into its components in the corresponding eigenspaces. Note that
\[ T\cdot \vf_\pm=\pm N\cdot \vf_\pm,\]
which, together with  (\ref{spin-calc-1}), implies that
\[ \la T\cdot \vf_+,\vf_-\ra 
=
\la  \vf_+,T\cdot\vf_-\ra 
=
-\la  \vf_+, N\cdot \vf_-\ra 
=
-\la N\cdot \vf_+,  \vf_-\ra 
=
-\la T\cdot \vf_+,  \vf_-\ra,\]
so that  $\la T\cdot \vf_+,\vf_-\ra=0$.
Then we use the fact (see \cite{Baum81} for a proof) that the hermitian form $(\phi,\psi)_T=\la T\cdot \phi,\psi\ra$ on $\Sigma$  is positive definite. The last equation then shows that $(\vf_+, \vf_-)_T=0$ and we get
\[
g(V_\vf,V_\vf)= - ( \vf,\vf)^2_T+  ( T\cdot N\cdot \vf,\vf)^2_T=
-4 (\vf_+,\vf_+)_T (\vf_-,\vf_-)_T\le 0.
\]
This shows that $V_\vf$ is either time-like or light-like.
\eprf
In fact, on a Lorentzian manifold  the  Dirac current of  spinor field is always causal even if the spinor is not parallel, but it may change its causal character from light-like to time-like. The proof of this has to take into account that $V_\vf$ may have zeros so that $N$ may not be well defined.

The following theorem gives a classification of Riemannian manifolds with imaginary Killing spinors.

\btheo[\cite{baum89-3}]\label{imspin}
Let $(M,g)$ be a complete Riemannian manifold with an imaginary Killing spinor. Then $(M,g)$ is globally isometric to hyperbolic space or to a warped product 
of the form 
\begin{equation}\label{warp} \left(\rr\times N ,  \d s^2+\e^{ 2 s} g_{N}\right),\end{equation}
where $(N,g_N)$ is a complete Riemannian manifold with a parallel spinor field. 
\etheo
\bprf
Let $(M,g)$ be a complete Riemannian manifold with an imaginary Killing spinor field. Then, by Theorem \ref{spincone}, the Lorentzian cone $(\hm,g_-)$ admits a parallel spinor field $\vf$, which by Lemma \ref{diraclemma} provides us with a parallel vector field $V_\vf$ that is  either light-like or time-like. In case it is time-like, Theorem \ref{Veps} yields that $(M,g)$ has constant sectional curvature $-1$ and hence is isometric to hyperbolic space. If $V_\vf$ is light-like, we can apply Theorem \ref{0pvf-theo} to get the desired warped product in (\ref{warp}) with a complete Riemannian manifold $(N,g_N)$. To get that $(N,g_N)$ admits a parallel spinor field we can either use the result in Remark \ref{spinwarp} or recall Corollary \ref{wavecor} and Theorem \ref{holtheointro} to obtain that the holonomy algebra of $(\hm,\hg_-)$ is equal to $\hol(N,g_N)\ltimes \rr^{\dim(N)}$. This is an indecomposable holonomy algebra that admits an invariant spinor under its spin representation if and only if $\hol(N,g_N)$ admits an invariant spinor.  
\eprf

The next theorem provides a classification of Lorentzian manifolds with real Killing spinors.

\btheo[\cite{bohle,baum00a}]
Let $(M,g)$ be a complete Lorentzian  manifold with a real Killing spinor. Then 
\bnum
\item  either $(M,g)$   is globally isometric to de Sitter space or  space or to a warped product 
of the form 
\begin{equation}\label{warp1} \left(\rr\times N ,  -\d s^2+\cosh^2(s) g_{N}\right),\end{equation}
where $(N,g_N)$ is a complete Riemannian manifold with a real Killing  spinor (i.e., with one of the structures in Theorem \ref{baertheo}), or
\item 
$M$ is a disjoint union $M=M_-\cup M_0 \cup M_+$ with  $M_0$   a smooth totally geodesic hypersurface and $M_\pm$ and such that 
 $(M_\pm,g)$ are globally isometric to
\[
(\rr\times N_{\pm} , - \d s^2+\e^{ 2 s} g_{N_{\pm}}),
\]
where $(N_{\pm} ,g_{N_{\pm}})$ are  complete Riemannian manifolds with parallel spinors.
\enum
\etheo
\bprf
If $(M,g)$ admits a real Killing spinor, then the cone $(\hm,\hg_+)$ admits a parallel spinor and hence  a parallel causal vector field $V$. 

If $V$ is time-like, then we apply Theorem \ref{-pvf-theo}, to get that $(M,g)$ is isometric to the Lorentzian manifolds in  (\ref{warp1}) with a complete Riemannian manifold $(N,g_N)$.  If $(N,g_N)$ is the round metric on the sphere then $(M,g)$ is de Sitter space.
The result in Remark \ref{spinwarp} shows that $(M,g)$ admits a real Killing spinor if and only if $(N,g_N)$ does. 

If $V$ is light-like, Theorem \ref{0pvf-theo} shows that (2) holds with complete Riemannian manifolds $(N_\pm,g_{N_\pm})$. To obtain that $(N,g_N)$ admits a parallel spinor, we use again Remark \ref{spinwarp} or recall Corollary \ref{wavecor} and Theorem \ref{holtheointro}, as for the proof of Theorem \ref{imspin}.
\eprf

\bibliographystyle{abbrv}
%

\providecommand{\MR}[1]{}\def\cprime{$'$} \def\cprime{$'$} \def\cprime{$'$}

\end{document}